\documentclass[letterpaper,reqno]{amsart}

\usepackage{float}
\usepackage{graphicx}

\usepackage[width=8.5in,height=11in,hmargin={1.0in,1.0in},vmargin={0.75in,0.75in}]{geometry}

\usepackage{mathrsfs}
\usepackage{hyperref}
\usepackage{color}

\theoremstyle{plain}
\newtheorem{problem}{Problem}
\newtheorem{theorem}{Theorem}
\theoremstyle{remark}
\newtheorem*{remark}{Remark}
\theoremstyle{definition}

\newtheorem{lemma}[theorem]{Lemma}
\newtheorem{corollary}[theorem]{Corollary}
\numberwithin{theorem}{section}

\newcommand{\OO}{\Omega}

\newcommand{\bff}{\boldsymbol{f}}

\newcommand{\bfn}{\boldsymbol{n}}
\newcommand{\bfu}{\boldsymbol{u}}
\newcommand{\bft}{\boldsymbol{t}}
\newcommand{\bfv}{\boldsymbol{v}}
\newcommand{\bfw}{\boldsymbol{w}}

\newcommand{\zero}{\boldsymbol{0}}

\newcommand{\ve}{\boldsymbol{\varepsilon}}

\newcommand{\II}{\textup{I}}

\newcommand{\PP}{\mathbb{P}}
\newcommand{\Gt}{\Gamma_S}
\newcommand{\GD}{\Gamma_D}

\newcommand{\bfH}{\boldsymbol{H}}

\newcommand{\trih}{\mathcal{T}_h}
\newcommand{\Eh}{\mathcal{E}_h}
\newcommand{\EE}{\mathcal{E}}
\newcommand{\sumkh}{\sum_{K\in\trih}}

\newcommand{\bfuh}{\boldsymbol{u}_h}

\newcommand{\bfvh}{\boldsymbol{v}_h}
\newcommand{\ph}{p_h}

\newcommand{\qh}{q_h}
\newcommand{\bfh}{\boldsymbol{h}}
\newcommand{\eu}{{e}^{\bfuh}}
\newcommand{\ep}{e^{\ph}}

\newcommand{\tbfu}{\bfu^*}
\newcommand{\tp}{p^*}

\newcommand{\etau}{{\eta}^{\bfuh}}
\newcommand{\etap}{\eta^{\ph}}

\newcommand{\nhh}[1]{|\!|\!| #1 |\!|\!|_h }	
\newcommand{\RR}{\mathbb{R}}
\newcommand{\NN}{\mathbb{N}}
\newcommand{\hK}{h_K}
\newcommand{\Cinv}{C_{\mathrm{in}}}
\newcommand{\Ctr}{C_{\mathrm{tr}}}

\newcommand{\bfIh}{\boldsymbol{I}_h}
\newcommand{\tbfIh}{\boldsymbol{\tilde I}_h}

\newcommand{\fch}[1]{{\color{black}{#1}}}
\newcommand{\alfonso}[1]{{\color{black}{#1}}}
\newcommand{\revision}[1]{{\color{black}{#1}}}
\numberwithin{equation}{section}

\begin{document}
\title[Stokes problem with slip boundary conditions]{Stokes problem with slip boundary conditions \\ using stabilized finite elements combined with Nitsche}
\author{Rodolfo Araya}
\address{Departamento de Ingenier{\'\i}a Matem\'atica \& CI2MA, Universidad de Concepci\'on, Concepci\'on, Chile}
\email{rodolfo.araya@udec.cl}
\author{Alfonso Caiazzo}
\address{Weierstrass Institute for Applied Analysis and Stochastics (WIAS), Mohrenstrasse 39, 10117 Berlin, Germany}
\email{alfonso.caiazzo@wias-berlin.de}
\author{Franz Chouly}
\address{University of the Republic, Faculty of Science, Center of Mathematics, 11400 Montevideo, Uruguay}
%
\email{fchouly@cmat.edu.uy}

\begin{abstract}
\fch{We discuss} how
slip conditions for the Stokes equation can be handled using Nitsche method, for a stabilized finite element discretization. Emphasis is made on the interplay between stabilization and Nitsche terms. Well-posedness of the discrete problem and optimal convergence rates, \revision{in natural norm for the velocity and the pressure}, are established, and illustrated with various numerical experiments. \revision{The proposed method fits naturally 
in the context of a finite element implementation while being accurate, and allows an increased flexibility in the choice of the finite element pairs.}
\end{abstract}


\maketitle
\section{Introduction}\label{S0}

Slip boundary conditions arise naturally for Stokes or Navier-Stokes equations, \fch{for instance when modelling biological surfaces \cite{B00}, in slide coating \cite{CS89} or in \alfonso{the context of} turbulence modeling \cite{MP94}.}
These are essential boundary conditions, and can be 
in fact considered as generalized Dirichlet conditions. 
These conditions are not straightforward to implement into standard finite element libraries, with the canonical techniques such as a discrete lifting or a partitioning of the global matrix \revision{especially for boundaries with variations of the unit outward normal vector}.
This challenge motivated several research works to study alternative approaches,
\revision{often based on penalty techniques or a mixed formulation. However, the penalty methods can yield inaccurate results or a poor conditioning if the penalty parameter is not chosen with care, and the mixed formulation can be cumbersome in this context, combined with the saddle-point structure already coming from the Stokes equations.}

This work presents a simple approach based on Nitsche's technique combined with a stabilized equal-order finite element method. 
To simplify the presentation, we focus on the Stokes equation on a polygonal boundary and without any specific law that involve the tangential components of the velocity, such as a Navier law. 
We consider both symmetric and non-symmetric variants of Nitsche, since they have different advantages, particularly to enforce accurately the boundary condition, \fch{see, \emph{e.g.},
\cite{C22,H18} and references therein}.
\alfonso{The stability analysis relies on the introduction of proper stabilization terms.} \fch{Notably, we are able to prove the stability with an inf-sup constant independent of the fluid viscosity.} The overall method is consistent, introduces no extra unknown and can be implemented easily. 
\alfonso{To assess the properties of the method},
we propose an implementation in the FEniCS environment \cite{ABHJKLRRRW15} and present several numerical experiments
\revision{that illustrate the simplicity, the flexibility and the accuracy of the method.}

Let us frame our work in \alfonso{a more general} perspective. The first methods to enforce slip conditions were based on Lagrange multipliers: see, e.g.,  \cite{L99,V87,V91}.
In \cite{BD99} the condition was enforced pointwise at nodal values of the velocity.
Many studies have been devoted to the study of penalty methods, to enforce approximately the slip condition with a regularization term. These methods are not consistent, but remain \alfonso{popular} 
and very easy to implement. 
\fch{Moreover, penalty can be interpreted as a penetration condition with a given resistance \cite{J02}.}
A first work has been focused on the Navier-Stokes equation \cite{CL09}, and followed by \cite{DTU13,DU15}, with emphasis on the case of a curved boundary, where a Babu\v{s}ka-type paradox may appear.
Other recent works have been devoted to the usage of penalty terms combined with Lagrange finite elements \cite{KOZ16,ZKO16,ZKO17} or Crouzeix-Raviart finite elements \cite{KOZ19,ZOK21}.
To our knowledge, Nitsche's method has been first considered in \cite{FS95}, as a simple, consistent and primal technique to take into account the slip condition. Notably, it has been noticed that the skew-symmetric variant of Nitsche remains operational even when the Nitsche parameter vanishes (\textit{penalty-free} variants), a result which opened the path to further research on this topic \cite{BCCLM18,BB16,B12}
 \revision{(see also \cite{BBH2009} for incompressible elasticity)}.
Later on, in \cite{UGF14}, different variants of Nitsche have been proposed and linked, as usual, with stabilized mixed methods (following \cite{S95}). Emphasis has been once again made on the curved boundary and a possibly related Babuska-type paradox.
\alfonso{More recently}, a specific treatment of the Navier boundary condition has been studied in \cite{WSMW18}, building on the specific Nitsche-type method proposed by Juntunen \& Stenberg \cite{JS09} to discretize robustly Robin-type boundary conditions (see also \cite{ZLR21}), and a symmetric Nitsche method with specific, accurate, discretization of the curved boundary, has been designed and studied in \cite{GS22}. \revision{Alternatively, an approach based on a specific design of wavelet functions has been proposed in \cite{harouna2021}.
In a recent work \cite{bansal2023nitsche}, a symmetric Nitsche method is combined with inf-sup stable pairs and a Variational MultiScale (VMS) stabilization for the Navier-Stokes equations. In \cite{gustafsson2023stabilised} a stabilized mixed method for a nonlinear slip condition is described.}
\alfonso{In conclusion, we observe} that almost all the aforementioned works have considered inf-sup stable pairs to discretize the Stokes equation, except \cite{KOZ16} where
\alfonso{the penalty method combined with 
a $\mathbb{P}_1/\mathbb{P}_1$ finite element pair with pressure stabilization is taken into account.}

This paper is structured as follows. \fch{Section \ref{sectmodel} describes the model equations in strong form. The weak formulation and the corresponding functional setting is object of Section \ref{S1}. Section \ref{S2} presents the discretization with finite elements, stabilization and Nitsche. Section \ref{S3} details the stability and convergence analysis. Numerical experiments are provided in Section \ref{S4}.
} 


\section{Model problem} \label{sectmodel}
Let $\Omega\subset\RR^{d}$, $d\in\{2,\,3\}$, be an open, bounded domain with Lipschitz continuous boundary  $\partial\Omega$. We  use standard notation for 
Lebesgue spaces $L^q(\OO)$, with norm $\|\cdot\|_{0,q,\OO}$, for $q>2$, and $\|\cdot\|_{0,\OO}$ for $q=2$ and inner product $(\cdot,\cdot)_\Omega$, and Sobolev spaces $H^m(\OO)$,  with norm $\|\cdot\|_{m,\OO}$ and semi-norm $|\cdot|_{m,\OO}$.
\fch{The boundary $\partial\Omega$ is partitionned into a subset $\Gamma_D$, where a Dirichlet boundary condition is imposed, with $\mathrm{meas}\, (\Gamma_D)>0$, and 
a subset $\Gamma_S$, where the slip condition is enforced.
Moreover, we denote with $\bfn$ the outer normal vector to $\Gamma_S$, and with $\bft_i,\; 1\leq i\leq d-1$ the orthonormal vectors spanning the plane tangent to $\Gamma_S$.}


We 
\alfonso{consider} the Stokes equations \alfonso{seeking for} a velocity field $\bfu : \OO \to \RR^d$ and a pressure field $p: \OO \to \RR$ solutions to
\begin{equation}\label{eq:stokes-strong}
\;\left\{
\begin{array}{rl}
-\nabla\cdot\sigma(\bfu,p) & = \bff \quad \textrm{ in } \Omega,    \\
\nabla \cdot \bfu & = 0 \quad\, \textrm{ in } \Omega,\\
\bfu &= \revision{\bfh} \quad \textrm{ on } \Gamma_D,\\
\bfu\cdot\bfn &= \revision{g} \quad\, \textrm{ on } \Gamma_S,\\
\sigma(\bfu,p)\bfn\cdot\bft_i &= s_i,\quad 1\leq i\leq d-1 \quad \textrm{ on } \Gamma_S.
\end{array}
\right.
\end{equation}
\alfonso{In \eqref{eq:stokes-strong},} \fch{the stress tensor is expressed as}
$$
\sigma(\bfu,p) := 2\nu\ve(\bfu)-p\II,
$$
the parameter $\nu > 0$ denotes the fluid viscosity,  $\ve(\bfu):= \frac{1}{2}(\nabla\bfu +\nabla\bfu^T)$
 stands for the symmetric part of \revision{the rate of deformation tensor $\nabla \bfu$}, $\bff \in L^2(\OO)^d$ is a given source term
 and $s_i\in L^2(\Gamma_S)$.

\section{Continuous variational formulation}\label{S1}
We define the Hilbert spaces
\begin{align*}
\bfH &:= \{\bfv\in H^1(\OO)^d\;:\;  \bfv = \zero\quad\mbox{on}\; \Gamma_D,\;\bfv\cdot\bfn = 0\quad\mbox{on}\; \Gamma_S\;\},\\
Q &:= L^2_0(\OO),
\end{align*}
\alfonso{with their natural inner products, and consider the product space $\bfH\times Q$ equipped 
with the norm
\[
\|(\bfv,q)\|^2 := \nu\|\ve(\bfv)\|^2_{0,\OO} + \|q\|^2_{0,\OO},\qquad \forall (\bfv,q)\in \bfH\times Q.
\]
We then introduce the bilinear forms 
\begin{equation*}
a: \bfH\times\bfH\to\RR, \; a(\bfu,\bfv) := 2\nu (\ve(\bfu),\ve(\bfv))_{\OO}\qquad \forall \bfu,\,\bfv\in \bfH,\\
\end{equation*}
and
\begin{equation*}
b:\bfH\times Q\to\RR, \; b(\bfv,q) := -(\nabla\cdot\bfv, q)_{\OO} \qquad \forall (\bfv,q)\in\bfH\times Q,
\end{equation*}
and consider the following variational formulation associated with problem \eqref{eq:stokes-strong}:} 
\begin{problem} \label{pbweak}
\noindent Find $(\bfu,p)\in \bfH\times Q$ such that
\begin{equation}\label{variational}
B((\bfu,p),(\bfv,q)) = F(\bfv,q),
\end{equation}
for all $(\bfv,q) \in \bfH\times Q$, where
\begin{equation}\label{eq:b-form}
B((\bfu,p),(\bfv,q)) := a(\bfu,\bfv) + b(\bfv,p) - b(\bfu,q)
\end{equation}
and
\[
F(\bfv,q):=  (\bff,\bfv)_{\OO}  + \sum_{i=1}^{d-1} \int_{\Gamma_S} s_i\,\bfv\cdot\bft_i\, \mathrm{d}s.
\]
\end{problem}

\noindent
Well-posedness of Problem~\ref{pbweak}
is stated below:

\begin{theorem}\label{well_cont}
Problem~\ref{pbweak} 
has a unique solution $(\bfu,p)\in \bfH\times Q$, and there exists a positive constant $C$
such that
\[
\|(\bfu,p)\| \leq C\left\{ \|\bff\|_{0,\OO} + \sum_{i=1}^{d-1} \|s_i\|_{0,\Gamma_S}\right\}.
\]
\end{theorem}
\begin{proof}
The proof is a direct consequence of the Babu\v{s}ka--Brezzi theory. \fch{See also \cite{B04}.}
\end{proof}
\section{Discrete stabilized scheme}\label{S2}
In what follows, we denote by $\{\trih\}_{h>0}$ a regular family of triangulations of $\bar\OO$ composed by simplices. For a given triangulation  $\trih$, we will denote by $\Eh$ the set of all faces (edges) of $\trih$, with the partitioning 
$$
\Eh:= \EE_{\OO} \cup \EE_D\cup \EE_S,
$$
where $\EE_{\OO}$ stands for the faces (edges) lying in the interior of $\OO$,  $\EE_S $ stands for the faces (edges) lying on the boundary $\Gamma_S$, and  $\EE_D$ stands for the edges (faces) lying on the boundary $\Gamma_D$. 
Moreover, we will denote with $K$ \alfonso{a generic element of a} triangulation  $\trih$, with $\hK$  the diameter of $K$ and 
define $h :=\displaystyle \max_{K\in\trih} \hK$.

\noindent 
\alfonso{As next, for a given $l\geq 1$,} we introduce the following finite element spaces:
\begin{align*}
	\bfH_{h}  := & \left\{ \bfv \in C(\overline{\OO})^d  \;:\; \bfv|_K \in  \PP_l(K)^d,\quad \forall K\in\trih  \right\},\\
	Q_{h}      := & \left\{  q \in C(\overline{\OO})  \;:\; q|_K \in  \PP_l(K), \quad  \forall K\in\trih \right\} \cap Q,
\end{align*} 
where $\PP_l$ stands for the  space of polynomials of total degree less or equal than $l$.
\begin{remark}
Note that $Q_h$ is a subspace of $Q$, but $\bfH_h$ is not a subspace of $\bfH$. In that sense,  imposing weakly the 
(slip or Dirichlet) boundary conditions
using Nitsche, \fch{can be considered a non-conforming finite element method.} 
\end{remark}

\noindent
In the sequel we will need the following well known results:
\begin{lemma}\label{inverse}[Inverse inequality]
Let $\bfvh\in\bfH_h$ then for each $K\in \trih$; $l, m\in\NN$, with $0\leq m\leq l$, there exists a positive constant $\Cinv$, independent of $K$, such that
\[
|\bfvh|_{l,K} \leq \Cinv\, h^{m-l}_K |\bfvh|_{m,K}.
\] 
\end{lemma}
\begin{proof}
See \cite[Lemma 12.1]{EG21}.
\end{proof}
\begin{lemma}\label{trace}[Trace inequality - I]
Let $\bfvh\in\bfH_h$ then for each $K\in \trih$, $E\subset \partial K$,  there exists a positive constant $\Ctr$, independent of $K$, such that
\[
\|\bfvh\|_{0,E} \leq \Ctr\, h^{-\frac{1}{2}}_K \|\bfvh\|_{0,K}.
\] 
\end{lemma}
\begin{proof}
See \cite[Lemma 12.8]{EG21}.
\end{proof}
\begin{lemma}\label{trace_b}[Trace inequality - II]
Let $v\in H^{1}(K)$, with $K\in\trih$. Then for any face (edge) $E\subset \partial K$, there exists a positive constant $\Ctr'$, independent of $K$, such that
\[
\|v\|_{0,E} \leq \Ctr' \left\{ h^{-1/2}_K \|v\|_{0,K} +  h^{1/2}_K \|\nabla v\|_{0,K}    \right\}
\]
\end{lemma}
\begin{proof}
Use \cite[Lemma 12.15]{EG21} and Young's inequality.
\end{proof}
\noindent
To define our stabilized scheme, we start by defining the following notation
\[
\revision{\langle \, \phi,\psi \rangle_{1/2,h,\EE_i} := \sum_{E\in \EE_i} \frac{1}{h_E} (\phi , \psi)_E},
\]
where $i=D\;\mbox{or}\; S$. 

Let $\theta \in \left\{ -1, 0, 1\right\}$. 
\revision{For a given Nitsche parameter $\gamma_0>0$ and a given stabilization parameter $\beta > 0$, the considered stabilized discrete scheme is given by:}  
\begin{problem}
Find $(\bfuh,\ph)\in\bfH_h\times Q_h$ such that
\begin{equation}\label{stab}
B_S((\bfuh,\ph),(\bfvh,\qh)) = F_S(\bfvh,\qh)\qquad \forall (\bfvh,\qh)\in\bfH_h\times Q_h,
\end{equation}
where
\begin{equation}\label{eq:B_S_h}
\begin{aligned}
B_S((\bfuh,\ph),(\bfvh,\qh)) :=&\,  B((\bfuh,\ph),(\bfvh,\qh))  - 2\nu(\ve(\bfuh)\bfn, \bfvh)_{\GD} - 2\theta\nu(\ve(\bfvh)\bfn, \bfuh)_{\GD}  \\
& + \nu \langle \gamma_0\,\bfuh, \bfvh \rangle_{1/2,h,\GD}   + (\ph, \bfvh\cdot\bfn)_{\GD} + \theta(\qh, \bfuh\cdot\bfn)_{\GD}\\
& -2\nu(\ve(\bfuh)\bfn\cdot\bfn, \bfvh\cdot\bfn)_{\Gt} - 2\theta\nu(\ve(\bfvh)\bfn\cdot\bfn, \bfuh\cdot\bfn)_{\Gt}  \\
& + \nu \langle \gamma_0\,\bfuh\cdot \bfn, \bfvh \cdot \bfn\rangle_{1/2,h,\Gt}   + (\ph, \bfvh\cdot\bfn)_{\Gt}
+ \theta(\qh, \bfuh\cdot\bfn)_{\Gt}
\\
& + \dfrac{\beta}{\nu} \sumkh h^2_K\,(-2\nu\,\nabla\cdot\ve(\bfuh) +\nabla \ph, \nabla \qh)_K\\
\end{aligned}
\end{equation}
and
\begin{equation}\label{eq:F_S_h}
\begin{aligned}
F_S(\bfvh,\qh) :=&\, F(\bfvh,\qh) -2\nu\theta ( \bfh, \ve(\bfvh)\bfn)_{\GD}+\theta ( \bfh\cdot\bfn, \qh)_{\GD} + \fch{\nu}\langle \gamma_0\,\bfh ,\bfvh \rangle_{1/2,h,\GD}\\
&-2\nu\theta ( g, \ve(\bfvh)\bfn\cdot\bfn)_{\Gt}+\theta ( g, \qh)_{\Gt} + \nu \langle \gamma_0\,g ,\bfvh \cdot \bfn\rangle_{1/2,h,\Gt}\\
& + \dfrac{\beta}{\nu} \sumkh h^2_K\,(\bff, \nabla \qh)_K\,.
\end{aligned}
\end{equation}
\end{problem}

\revision{The bilinear form \eqref{eq:B_S_h} contains, besides
the standard terms of the weak form of the Stokes equations, 
the Nitsche terms on the Dirichlet boundary ($\Gamma_D$, first and second line), the Nitsche terms on the slip boundary 
($\Gamma_S$, third and fourth lines), and the 
stabilization term (fifth line).}

\revision{The different choices of the parameter $\theta$ allow to recover three variants of the method 
as for Nitsche for Dirichlet boundary condition~\cite{C22} and as for the discontinuous Galerkin Interior Penalty (dGIP) method~\cite{dipietro2012}:
\begin{enumerate}
\item For $\theta=1$, the method is symmetric in the spirit of Nitsche's original formulation~\cite{N71}. 
It can be built alternatively from an augmented Lagrangian formalism, see \cite[Section 5.2.2]{burman2022}. 
\item For $\theta=0$, we get the simplest, incomplete, formulation, that has the less terms, and that is presented for instance in \cite[Section 37.1]{EG21b}.
\item For $\theta=-1$, we recover the skew-symmetric formulation
of J. Freund and R. Stenberg~\cite{FS95}, where discrete ellipticity is ensured for any $\gamma_0 > 0$.
\end{enumerate}
}


\noindent
\fch{We state below a consistency result:}
\begin{lemma}[Consistency]\label{consist}
Let $(\bfu,p)\in \bfH\times Q$ and $(\bfuh,\ph)\in\bfH_h\times Q_h$ be the solutions to Problem \eqref{variational} and Problem \eqref{stab}, respectively. Assume that $(\bfu,p)\in (H^2(\OO)^d\cap \bfH) \times(H^1(\OO)\cap Q)$,
then 
\[
B_S((\bfu -\bfuh, p-\ph),(\bfvh,\qh)) = 0\qquad\forall  (\bfvh,\qh)\in \bfH_h\times Q_h.
\]
\end{lemma}
\begin{proof}
The proof is a direct consequence of the definition of 
\revision{Problem \eqref{variational} and Problem \eqref{stab}},
combined with the regularity assumptions.
\end{proof}

\noindent
Over $\bfH_h\times Q_h$, we consider the discrete norm
\[
\nhh{(\bfvh,\qh)} := 	\left({\nu} \|\ve(\bfvh)\|^2_{0,\OO}  + \sum_{E\in \EE_D} \frac{{\nu}}{h_E} \| \bfvh\|^2_{0,E}+  \sum_{E\in \EE_S} \frac{{\nu}}{h_E} \| \bfvh\cdot\bfn\|^2_{0,E} + \sumkh \frac{h^2_K}{{\nu}}\, \|\nabla \qh\|^2_{0,K}\right)^{1/2}\,.
\]
We state below the continuity of the discrete bilinear form.
\begin{theorem}\label{cont}
For $\theta = -1, 0, 1$, there exists a positive constant $C_a$, independent of $h$, $\nu$, and $\theta$, such that
\[
B_S((\bfuh,\ph),(\bfvh,\qh)) \leq C_a\, \|(\bfuh,\ph)\|\,\|{(\bfvh,\qh)}\| \qquad \forall \,(\bfuh,\ph),\,(\bfvh,\qh)\in \bfH_h\times Q_h.
\]
\end{theorem}
\begin{proof}
\fch{This is} \alfonso{as a direct consequence of}
lemmas \ref{inverse}, \ref{trace}, Cauchy-Schwarz and H\"older inequalities. 
\end{proof}

\noindent
The well-posedness of our finite element discretization is established as follows.

\begin{theorem}[Well-posedness]
\label{well}
For $\theta = -1, 0, 1$,
\fch{for $\gamma_0$ large enough and for $\beta$ small enough, there exists}
%
a positive constant  $C_S = C_S(\theta,\beta,\gamma_0)$, 
independent on $h$ and $\nu$,
such that,
\[
B_S((\bfuh,\ph),(\bfuh,\ph)) \geq C_S \nhh{(\bfuh,\ph)}^2 \qquad \forall \,\,(\bfuh,\ph)\in \bfH_h\times Q_h\,.
\]
\revision{As a result, Problem \eqref{stab} admits a unique solution.}
The \fch{bounds on the }parameters $\beta$ and $\gamma_0$ depend only
on the trace and inverse inequality constants. Moreover,
in the skew-symmetric case $\theta=- 1$, well-posedness can be proven
for any $\gamma_0>0$.
\end{theorem}
\begin{proof}\ \\
\noindent\underline{$\theta=-1$.} 
Take $(\bfuh,\ph)\in \bfH_h\times Q_h$.
Using Lemma \ref{inverse}, H\"older and Young inequalities,  we get that
\begin{align}
2\beta h^2_K\,(\nabla\cdot\ve(\bfuh), \nabla \ph)_K &\leq 
2\beta h^2_K\,\|\nabla\cdot\ve(\bfuh)\|_{0,K}\, \|\nabla \ph\|_{0,K}\nonumber\\
& = 2\beta \,\nu^{\frac12} \|\nabla\cdot\ve(\bfuh)\|_{0,K}\, h^2_K \nu^{-\frac12}\|\nabla \ph\|_{0,K}\nonumber\\
& \leq 2\beta \left( \frac{\delta_1}{2} \Cinv^2 \,\fch{\nu} 
\|\ve(\bfuh)\|^2_{0,K}
+ \frac{1}{2\delta_1} \frac{h^2_K}{\nu} \|\nabla \ph\|^2_{0,K}\right)\nonumber\\
&\leq \beta\,\Cinv^2\,\delta_1\,\nu \|\ve(\bfuh)\|^2_{0,K} +  \frac{\beta}{\delta_1} \, \dfrac{h^2_K}{\nu}\|\nabla \ph\|^2_{0,K},\label{eq0}
\end{align}
with $\delta_1$ a positive parameter to be chosen in a convenient way. Now, using \eqref{eq0} and the definition of $B_S$, with $\theta=-1$, we obtain
\begin{align*}
&B_S((\bfuh,\ph),(\bfuh,\ph)) = 2\nu\|\ve(\bfuh)\|^2_{0,\OO} + \fch{\nu}\langle \gamma_0\,\bfuh, \bfuh \rangle_{1/2,h,\GD} + \fch{\nu}\langle \gamma_0\,\bfuh\cdot \bfn, \bfuh \cdot \bfn\rangle_{1/2,h,\Gt} \\
&+ \dfrac{\beta}{\nu} \sumkh h^2_K\,(-2\nu\,\nabla\cdot\ve(\bfuh) +\nabla \ph, \nabla \ph)_K\\
&= 2\nu\|\ve(\bfuh)\|^2_{0,\OO} + \nu \sum_{E\in\EE_D}  \frac{\gamma_0}{h_E} \| \bfuh\|^2_{0,E}  + \nu \sum_{E\in \EE_S} \frac{\gamma_0}{h_E} \| \bfuh\cdot\bfn\|^2_{0,E}  \\
& \quad \quad  +\dfrac{\beta}{\nu} \sumkh h^2_K\,\|\nabla \ph\|^2_{0,K} - 2\beta \sumkh h^2_K\,(\nabla\cdot\ve(\bfuh), \nabla \ph)_K\\
&\geq (2- \beta\,\Cinv^2\,\delta_1) \nu \|\ve(\bfuh)\|^2_{0,\OO} + \nu \sum_{E\in \EE_D} \frac{\gamma_0}{h_E} \| \bfuh\|^2_{0,E} + \nu \sum_{E\in \EE_S} \frac{\gamma_0}{h_E} \| \bfuh\cdot\bfn\|^2_{0,E} \\
&+ \beta \left(1 - \frac{1}{\delta_1}\right) \,\sumkh \frac{h^2_K}{\nu}\,\|\nabla \ph\|^2_{0,K}.
\end{align*}
Now, choosing $\delta_1=2$ and the stabilization parameter $\beta$ such that
$\beta<\Cinv^{-2}$  we obtain that there exists a positive constant 
\begin{equation}\label{eq:C_S}
C_S = \text{min}\left\{ 2(1-\beta \Cinv^2),\gamma_0,\frac{\beta}{2}\right\}\,,
\end{equation}
independent of $h$ \fch{and $\nu$}, such that
\[
B_S((\bfuh,\ph),(\bfuh,\ph)) \geq C_S \nhh{(\bfuh,\ph)}^2,
\]
which proves that the problem is well posed.\\

\noindent\underline{$\theta =0$.} Using Lemma \ref{trace}, H\"older and Young inequalities, we have that
\begin{align}
2\nu(\ve(\bfuh)\bfn\cdot\bfn, \bfuh\cdot\bfn)_{\Gt} &\leq 2\nu \sum_{E\in \EE_S} \|\ve(\bfuh)\bfn\|_{0,E}\,\| \bfuh\cdot\bfn\|_{0,E}\nonumber\\
&= 2\nu\, \sum_{E\in \EE_S} h^{1/2}_E\, \|\ve(\bfuh)\bfn\|_{0,E}\,h^{-1/2}_E\| \bfuh\cdot\bfn\|_{0,E}\nonumber\\
&\leq  2\nu\, \sum_{E\in \EE_S}\left( \frac{h_E\delta_2}{2}\, \|\ve(\bfuh)\|^2_{0,E}\ + \frac{h^{-1}_E}{2\delta_2}\| \bfuh\cdot\bfn\|^2_{0,E}\right)\nonumber\\
&\leq  \nu\delta_2 \Ctr^2\, \sum_{K\in \trih} \, \|\ve(\bfuh)\|^2_{0,K}\ 
+ \revision{\frac{\nu}{\delta_2} \sum_{E\in \EE_S} \frac{1}{h_E}\| \bfuh\cdot\bfn\|^2_{0,E}}\nonumber\\
&\leq  \delta_2 \Ctr^2\, \nu \|\ve(\bfuh)\|^2_{0,\OO}\ + \frac{\nu}{\delta_2}  \langle\,\bfuh\cdot \bfn, \bfuh \cdot \bfn\rangle_{1/2,h,\Gt}.\label{eq1}
\end{align}
Using again the same arguments it is easy to prove, for any $\delta_3>0$, that
\begin{equation}
2\nu(\ve(\bfuh)\bfn, \bfuh)_{\GD} \leq\delta_3 \Ctr^2\,   \nu\|\ve(\bfuh)\|^2_{0,\OO}\ + \frac{\nu}{\delta_3}  \langle \,\bfuh, \bfuh \rangle_{1/2,h,\GD}.\label{eq1b}
\end{equation}
On the other hand, using lemmas \ref{inverse} and \ref{trace}, and H\"older's inequality, we have that
\begin{align}
(\ph, \bfuh\cdot\bfn)_{\Gt} &\leq  \sum_{E\in \EE_S} \|\ph\|_{0,E}\,\|\bfuh\cdot\bfn\|_{0,E}\nonumber\\
&= \sum_{E\in \EE_S} \nu^{-\frac12}h^{1/2}_E\|\ph\|_{0,E}\, \nu^{\frac12} h^{-1/2}_E\|\bfuh\cdot\bfn\|_{0,E}\nonumber\\
&\leq \sum_{E\in \EE_S} \frac{h_E\delta_4}{2 \nu} \|\ph\|^2_{0,E} + \sum_{E\in \EE_S} \nu \frac{h^{-1}_E}{2\delta_4}\|\bfuh\cdot\bfn\|^2_{0,E}\nonumber\\
&\leq \sum_{K\in \trih} \frac{\Ctr^2\delta_4}{2 \nu} \|\ph\|^2_{0,K} +  \frac{1}{2\delta_4}\sum_{E\in \EE_S} \frac{\nu}{h_E}\|\bfuh\cdot\bfn\|^2_{0,E}\nonumber\\
&\leq \frac{\Cinv^2\,\Ctr^2\delta_4}{2}\,\sum_{K\in \trih} \, \frac{h^2_K}{\nu} \|\nabla\ph\|^2_{0,K} + \frac{\nu}{2\delta_4}  \langle \,\bfuh\cdot \bfn, \bfuh \cdot \bfn\rangle_{1/2,h,\Gt}\label{eq2}.
\end{align}
In the same way, we can prove that
\begin{equation}\label{eq2b}
(\ph, \bfuh\cdot\bfn)_{\GD} \leq \frac{\Cinv^2\,\Ctr^2\delta_5}{2}\,\sum_{K\in \trih} \, \frac{h^2_K}{\nu} \|\nabla\ph\|^2_{0,K} + \frac{\nu}{2\delta_5}  \langle\bfuh, \bfuh \rangle_{1/2,h,\GD}.
\end{equation}
Thus, using \eqref{eq0}--\eqref{eq2b}, and Young's inequality we get
\begin{align*}
&B_S((\bfuh,\ph),(\bfuh,\ph)) = 2\nu\|\ve(\bfu)\|^2_{0,\OO} -  2\nu(\ve(\bfuh)\bfn, \bfuh)_{\GD}  + 
\fch{\nu}
\langle \gamma_0\,\bfuh, \bfuh \rangle_{1/2,h,\GD} + (\ph, \bfuh\cdot\bfn)_{\GD}\\
&-2\nu(\ve(\bfuh)\bfn\cdot\bfn, \bfuh\cdot\bfn)_{\Gt}  + \fch{\nu} \langle \gamma_0\,\bfuh\cdot \bfn, \bfuh \cdot \bfn\rangle_{1/2,h,\Gt}+ (\ph, \bfuh\cdot\bfn)_{\Gt}
\\ &
+ \dfrac{\beta}{\nu} \sumkh h^2_K\,(-2\nu\,\nabla\cdot\ve(\bfuh) +\nabla \ph, \nabla \ph)_K\\
&\geq (2 -\delta_2\Ctr^2  -\delta_3 \Ctr^2 -\beta \Cinv^2\delta_1) \nu \|\ve(\bfu)\|^2_{0,\OO} + 
\left(\gamma_0-\frac{1}{\delta_3} - \frac{1}{2\delta_5}\right)  \nu \langle \,\bfuh, \bfuh\rangle_{1/2,h,\GD} \\
&+ \left(\gamma_0-\frac{1}{\delta_2} - \frac{1}{2\delta_4}\right)   \nu \langle \,\bfuh\cdot \bfn, \bfuh \cdot \bfn\rangle_{1/2,h,\Gt}+ 
\left(
\beta -\frac{\beta}{\delta_1} -\frac{\Cinv^2\Ctr^2\delta_4}{2} -\frac{\Cinv^2\Ctr^2\delta_5}{2}
\right) \,\sumkh \frac{h^2_K}{\nu}\,\|\nabla \ph\|^2_{0,K}.
\end{align*}
The positive parameters $\delta_1$, $\delta_2$, 
$\delta_3$, $\delta_4$, and $\delta_5$ should be
now chosen properly. We  take $\delta_1 = 2$, $\delta_2=\delta_3$, $\delta_4 = \delta_5$, obtaining
$$
\begin{aligned}
B_S((\bfuh,\ph),(\bfuh,\ph)) \geq & \ 2\left(1 - \delta_2 \Ctr^2- \beta \Cinv^2\right) \nu \|\ve(\bfu)\|^2_{0,\OO} 
+ 
\left( \gamma_0-\frac{1}{\delta_2} - \frac{1}{2\delta_4}\right)  \nu \langle \,\bfuh, \bfuh\rangle_{1/2,h,\GD} \\
& + \left( \gamma_0-\frac{1}{\delta_2} - \frac{1}{2\delta_4}\right)   \nu \langle \,\bfuh\cdot \bfn, \bfuh \cdot \bfn\rangle_{1/2,h,\Gt} \\
& + \left(\frac{\beta}{2} -\delta_4 \Cinv^2\Ctr^2
\right) \,\sumkh \frac{h^2_K}{\nu}\,\|\nabla \ph\|^2_{0,K}.
\end{aligned}
$$
Choosing $\delta_2 = \frac{\beta}{\Ctr^2}$, $\delta_4 = \frac{\beta}{4 \Ctr^{2} \Cinv^{2}}$ one gets
$$
\gamma_0 - \frac{1}{\delta_2}  - \frac{1}{2\delta_4}= \gamma_0 - 
\Ctr^2\frac{1+ 2 \Cinv^{2}}{\beta},
$$
$$
\frac{\beta}{2} -\delta_4 \Cinv^2\Ctr^2 = \frac{\beta}{4},
$$
$$
1 - \delta_2 \Ctr^2- \beta \Cinv^2 = 1 - \beta(\Ctr^2 + \Cinv^2).
$$
Hence, taking $\beta < (\Ctr^2 + \Cinv^2)^{-1}$ and $\gamma_0 > \Ctr^2\frac{1+ 2 \Cinv^{2}}{\beta}$
yields
\[
B_S((\bfuh,\ph),(\bfuh,\ph)) \geq C_S \nhh{(\bfuh,\ph)}^2,
\]
with 
\begin{equation}
C_S:= \min\left\{2\left(1 - \beta(\Ctr^2 + \Cinv^2)\right), \frac{\beta}{4}, \gamma_0 - 
\Ctr^2\frac{1+ 2 \Cinv^{2}}{\beta}\right\}\,.
\end{equation}
The case $\theta=1$ is proved using  the same arguments as for $\theta=0$. 
\end{proof}
\section{Error analysis} \label{S3}
\alfonso{This section is devoted to} a priori error analysis based on the arguments of \cite{BF02} and \cite{BB16}. To this purpose, let $\bfIh:\bfH\to\bfH_h$ and $J_h: Q\to Q_h$ be the vectorial and scalar version of the Scott–Zhang interpolant, respectively. Then the following results concerning the approximation properties of these operators hold.
\begin{lemma}\label{SZ}
For each $K\in \trih$ and $k\geq 1$, there exist two positive constants $C_{SZ}$ and $\tilde{C}_{SZ}$, independent of $h_K$, such that
\begin{equation}\label{eq:SZ}
\begin{aligned}
\|\bfu - \bfIh \bfu\|_{0,K} + h_K |\bfu - \bfIh \bfu|_{1,K} +   h^2_K |\bfu - \bfIh \bfu|_{2,K} &\leq C_{SZ}\, h^{k+1}_K \fch{|\bfu|_{k+1,\omega_K}}\qquad \forall \bfu\in H^{k+1}(\omega_K)^d,\\
\|q - J_h q\|_{0,K} + h_K |q - J_h q|_{1,K}  &\leq \tilde{C}_{SZ}\, h^{k}_K \fch{|q|_{k,\omega_K}}\qquad \forall q\in H^{k}(\omega_K),\\
\end{aligned}
\end{equation}
where the patch $\omega_K$ is defined as
\[
\omega_K := \displaystyle{\bigcup_{ K\cap  K'\ne \emptyset} K'}.
\]
\end{lemma}
\begin{proof}
See \cite{SZ90}.
\end{proof}
\begin{remark}
For a given element $q\in Q$, the Scott–Zhang interpolant $J_h q $ does not belong, in general, to $Q=L^2_0(\OO)$. However, we can consider its modified version given by
\[
J_h q - \frac{1}{|\OO|} \int_\OO J_h q \,\mathrm{d}x \in Q.
\]
that\alfonso{, with a little abuse of notation}, we will also denote as $J_h q$. 
It still satisfies \eqref{eq:SZ}. 
\end{remark}

\noindent
Let us introduce the following notation
\begin{align*}
\eu &:= \bfIh \bfu - \bfuh ,\qquad  \ep :=  J_h p -\ph \\
\etau&:= \bfu - \bfIh\bfu,\qquad  \etap := p - J_h p.
\end{align*}
Note that
\[
\bfu -\bfuh = \etau + \eu \qquad\mbox{and}\qquad p-\ph = \etap+ \ep.
\]


\noindent \fch{The following theorem states the convergence of the method.}

\begin{theorem}\label{apriori}
Let $(\bfu,p)\in \bfH\times Q$ and $(\bfuh,\ph)\in\bfH_h\times Q_h$ be the solutions to Problem \eqref{variational} and Problem \eqref{stab}, respectively. Assume that $(\bfu,p)\in (H^{k+1}(\OO)^d\cap \bfH) \times(H^k(\OO)\cap Q)$,
with $k\geq 1$, then there exists a constant $C_E>0$ such that
\begin{equation}\label{eq:error-estim}
\nhh{(\bfu -\bfuh,p -\ph)} \leq C_E h^k \{|\bfu|_{k+1,\OO} + |p|_{k,\OO}\}.
\end{equation}
\end{theorem}
\begin{proof}
For any $K\in \trih$ and $E\subset \partial K$, we have, as a direct consequence of lemmas \ref{trace_b} and \ref{SZ}, that 
\fch{
\begin{align}
\|\ve(\etau)\|_{0,K} &\leq C_1 h^k_K |\bfu|_{k+1,K}\label{eq_a},\\
h^{1/2}_E\|\ve(\etau)\bfn\|_{0,E} &\leq C_1 h^k_K |\bfu|_{k+1,K}\label{eq_b},\\
h^{-1/2}_E\|\etau\cdot\bfn\|_{0,E} &\leq C_1 h^k_K |\bfu|_{k+1,K}\label{eq_c},\\
\|\nabla\cdot \etau\|_{0,K} &\leq C_1 h^k_K |\bfu|_{k+1,K}\label{eq_d},\\
h_K\|\nabla\cdot\ve(\etau)\|_{0,K} &\leq C_1 h^k_K |\bfu|_{k+1,K}\label{eq_e},\\
h_K\|\nabla\etap\|_{0,K} &\leq C_1 h^k_K |p|_{k,K}\label{eq_f},
\end{align}}
for a positive constant $C_1>0$ depending on the constants $\Ctr'$ introduced in Lemma \ref{trace_b}, $C_{SZ}$, and $\tilde{C}_{SZ}$.

Now, using Cauchy-Schwarz's inequality and \eqref{eq_a}--\eqref{eq_f}, we obtain that there exists a constant $C_2= C_2(\nu^{\frac12},\theta,\beta,\gamma_0)$ such that
\begin{align}
B_S((\etau,\etap),(\bfvh,\qh)) := & \; 2\nu (\ve(\etau),\ve(\bfv_h))_{\OO}   -(\nabla\cdot\bfv_h, \etap)_{\OO} + (\nabla\cdot\etau, q_h)_{\OO}\nonumber\\
& - 2\nu(\ve(\etau)\bfn, \bfvh)_{\GD} - 2\theta\nu(\ve(\bfvh)\bfn, \etau)_{\GD} \nonumber \\
& + \fch{\nu} \langle \gamma_0\,\etau, \bfvh \rangle_{1/2,h,\GD} + \theta(\qh, \etau\cdot\bfn)_{\GD}  + (\etap, \bfvh\cdot\bfn)_{\GD}\nonumber\\
& -2\nu(\ve(\etau)\bfn\cdot\bfn, \bfvh\cdot\bfn)_{\Gt} - 2\theta\nu(\ve(\bfvh)\bfn\cdot\bfn, \etau\cdot\bfn)_{\Gt} \nonumber \\
& + \fch{\nu} \langle \gamma_0\,\etau\cdot \bfn, \bfvh \cdot \bfn\rangle_{1/2,h,\Gt} + \theta(\qh, \etau\cdot\bfn)_{\Gt}  + (\etap, \bfvh\cdot\bfn)_{\Gt}\nonumber\\
& + \dfrac{\beta}{\nu} \sumkh h^2_K\,(-2\nu\,\nabla\cdot\ve(\etau) +\nabla \etap, \nabla \qh)_K \nonumber\\
\fch{\leq C_2}
& \; h^k \{|\bfu|_{k+1,\OO} + |p|_{k,\OO}\}\, \nhh{(\bfvh,\qh)}.\label{cota_a}
\end{align}
On the other hand, using again \eqref{eq_a}-- \eqref{eq_f}, and the definition of $\nhh{\cdot}$, we have that
\begin{equation}\label{cota_b0}
\begin{aligned}
\nhh{(\etau,\etau)}^2 & = {\nu} \|\ve(\etau)\|^2_{0,\OO}  + \sum_{E\in \EE_D} \frac{{\nu}}{h_E} \| \etau\|^2_{0,E}+  \sum_{E\in \EE_S} \frac{{\nu}}{h_E} \| \etau\cdot\bfn\|^2_{0,E} + \sumkh \frac{h^2_K}{{\nu}}\, \|\nabla \etau\|^2_{0,K} 
\end{aligned}
\end{equation}
which yields
\begin{equation}\label{cota_b}
\begin{aligned}
\nhh{(\etau,\etau)} & \leq 
3\nu^{\frac12} C_1 \nu h^k |\bfu|_{k+1,\OO}
+\nu^{-\frac12} C_1 |p|_{k,\OO} \\
& \leq C_3 h^k \{|\bfu|_{k+1,\OO} + |p|_{k,\OO}\}
\end{aligned}
\end{equation}
with $C_3 = C_1\max\left\{3\nu^{\frac12},\nu^{-\frac12}\right\}$
Thus, from Theorem \ref{well}, Lemma \ref{consist} and \eqref{cota_a}, we get
\begin{align}
\nhh{(\eu,\ep)}^2 & 
\fch{\leq C_S} B_S((\eu,\ep),(\eu,\ep)) = - \alfonso{C_S}B_S((\etau,\etap),(\eu,\ep))\nonumber\\
&
\alfonso{\leq C_S C_2} h^k \{|\bfu|_{k+1,\OO} + |p|_{k,\OO}\}\,\nhh{(\eu,\ep)}.\label{cota_c}
\end{align}
The estimate \eqref{eq:error-estim} follows using triangle inequality and the bounds  \eqref{cota_b} and \eqref{cota_c},
with $C_E = C_3 + C_S C_2$\,.
\end{proof}
\noindent
{Concerning the error in the natural norm of the pressure we have the following result.
\begin{corollary}
Let assume the hypothesis of Theorem \ref{apriori} are satisfied and let us assume that the family of triangulations $\{\trih\}_{h>0}$ is quasi-uniform, then
\[
\|p-p_h\|_{0,\OO} \preceq h^k \{|\bfu|_{k+1,\OO} + |p|_{k,\OO}\}.
\]
\end{corollary}}
\begin{proof}
We will follow the arguments of \cite[Theorem 1]{BD88}.

\noindent
Using Theorem \ref{apriori} and the quasi-uniformity assumption we obtain 
\begin{align}
&\left({\nu} \|\ve(\bfu-\bfuh)\|^2_{0,\OO}  + \sum_{E\in \EE_D} \frac{{\nu}}{h_E} \| \bfu -\bfuh\|^2_{0,E}+  \sum_{E\in \EE_S} \frac{{\nu}}{h_E} \| \bfu\cdot\bfn-\bfuh\cdot\bfn\|^2_{0,E} +  \frac{h^2}{{\nu}}\, |p-p_h|^2_{1,\OO}\right)^{1/2}\nonumber\\
 &\leq C_E  h^k \{|\bfu|_{k+1,\OO} + |p|_{k,\OO}\}.\label{eq2c}
\end{align}
To address the convergence of the pressure, we introduce $(\tbfu,\tp)$ as the solution of the auxiliary problem below:
\begin{equation}\label{eq:auxiliar}
\;\left\{
\begin{array}{rl}
-\nabla\cdot\sigma(\tbfu,\tp) & = \zero \quad \textrm{ in } \Omega,    \\
\nabla \cdot \tbfu & = p-p_h \quad\, \textrm{ in } \Omega,\\
\tbfu &= \zero \quad \textrm{ on } \Gamma_D,\\
\tbfu\cdot\bfn &= 0 \quad\, \textrm{ on } \Gamma_S\\
\sigma(\tbfu,\tp)\bfn\cdot\bft_i &= 0,\quad 1\leq i\leq d-1 \quad \textrm{ on } \Gamma_S.
\end{array}
\right.
\end{equation}
From the standard theory, it  is clear that there exist $C^* = C^*(\nu)>0$ such that
\begin{equation}\label{cota1}
\|\tbfu\|_{1,\OO} + \|\tp\|_{0,\OO} \leq C^* \|p-p_h\|_{0,\OO}.
\end{equation}
Let $\tbfIh:\bfH\to\bfH_h\cap H^1_0(\OO)^d$ be a  version of the Scott–Zhang interpolant, let $\bfw_h:= \tbfIh \tbfu$, and let $C_{SZ}^*$ be the maximum between the corresponding stability and approximation constants.

Using \eqref{eq:auxiliar}, \eqref{eq:stokes-strong}, \eqref{eq:B_S_h}, and integration by parts we obtain
\begin{equation}\label{eq:p-estim-0}
\begin{aligned}
\|p-p_h\|^2_{0,\OO} &= (p-p_h,\nabla\cdot \tbfu)= (p-p_h,\nabla\cdot \bfw_h) + (p-p_h,\nabla\cdot(\tbfu -\bfw_h)) = \\ 
& = a(\bfu -\bfu_h,\bfw_h) +  2\theta\nu(\ve(\bfw_h)\bfn, \bfuh -\bfh)_{\GD}
+ 2\theta\nu(\ve(\bfw_h)\bfn\cdot\bfn, \bfuh\cdot\bfn - g)_{\Gt} \\
& \quad - (\nabla(p-p_h),\tbfu-\bfw_h)\,.
\end{aligned}
\end{equation}
Next, using the quasi-uniformity of the mesh family and the stability and approximation properties of the Scott-Zhang interpolant allow to obtain
\begin{equation}\label{eq:p-estim-1}
\begin{aligned}
& \|p-p_h\|^2_{0,\OO} \\ 
\leq & \nu \|\ve(\bfu-\bfu_h)\|_{0,\OO} \|\ve(\bfw_h)\|_{0,\OO} \\ &  + 2\nu^{\frac12}\theta \sum_{E\in \EE_D} h^{\frac12}_E\,\|\ve(\bfw_h)\|_{0,E}\left(\frac{\nu^{\frac12}}{h_E^{\frac12}}\|\bfu -\bfuh\|_{0,E}\right) \\
&  +2\nu^{\frac12}\theta \sum_{E\in \EE_S} h^{\frac12}_E\|\ve(\bfw_h)\|_{0,E}\left(\frac{\nu^{\frac12}}{h_E^{\frac12}}\|\bfu\cdot\bfn -\bfuh\cdot\bfn\|_{0,E}\right) \\
 & + C_{SZ}^* h|p-p_h|_{1,\OO}\|\tbfu\|_{1,\OO} \\
 \leq &C_P\left({\nu} \|\ve(\bfu-\bfuh)\|^2_{0,\OO}  + \sum_{E\in \EE_D} \frac{{\nu}}{h_E} \| \bfu -\bfuh\|^2_{0,E}+  \sum_{E\in \EE_S} \frac{{\nu}}{h_E} \| \bfu\cdot\bfn-\bfuh\cdot\bfn\|^2_{0,E} +  \frac{h^2}{{\nu}}\, |p-p_h|^2_{1,\OO}\right)^{\frac12} \|\tbfu\|_{1,\OO}\,,\\
\end{aligned}
\end{equation}
with $C_P = \nu^{\frac12} C_{SZ}^*(1+4 \theta)$.
The results follows using Lemma \ref{inverse}, \eqref{eq2c}, and \eqref{cota1}:
\begin{equation}\label{eq:p-estim-2}
\begin{aligned}
\|p-p_h\|^2_{0,\OO} & \leq\,C_P\,C_E \,C^*\, h^k \{|\bfu|_{k+1,\OO} + |p|_{k,\OO}\} \,\|p-p_h\|_{0,\OO}\,.
\end{aligned}
\end{equation}
\end{proof}

\section{Numerical experiments} \label{S4}

Numerical experiments are carried out with FEniCS finite element software~\cite{logg2012}, and the scripts are available online \cite{Araya2023}.

\subsection{Example 1: 2D Cavity}
\revision{The value of $s_1$ is calculated from the explicit solution and also a nonhomogeneous Dirichlet boundary condition is used.} 
In this example, taken from \cite{UGF14}, we 
\fch{take $\Omega:=(-1,1)^2$ and the slip boundary condition is imposed on $y=-1$, while a Dirichlet boundary condition is enforced on the rest of the boundary}.
The exact solution of this problem is given by $\bfu:= (2y(1-x^2), -2x(1-y^2))$ and $p:=0$. For our computations we use $\PP_1$ for all the variables, while the viscosity is set to $\nu=1$.
\fch{The corresponding computed velocity field is depicted Figure~\ref{Fig2}.
Table~\ref{tableEx1} presents the approximation errors for pressure and velocity as well as the computed convergence rate, which are in good agreement with the theory, with a slight superconvergence for the pressure. Table~\ref{tab2} presents the error in $L^2$ norm on the slip condition on $\Gamma_S$, that, for this situation, does not differ significantly between the symmetric and skew-symmetric variants. However, it shows that the larger the Nitsche parameter $\gamma_0$ is, the smaller is the error on the slip condition, for both variants.}

\begin{figure}[H]
	\centering
	\includegraphics[width=10.0cm,height=10.0cm]{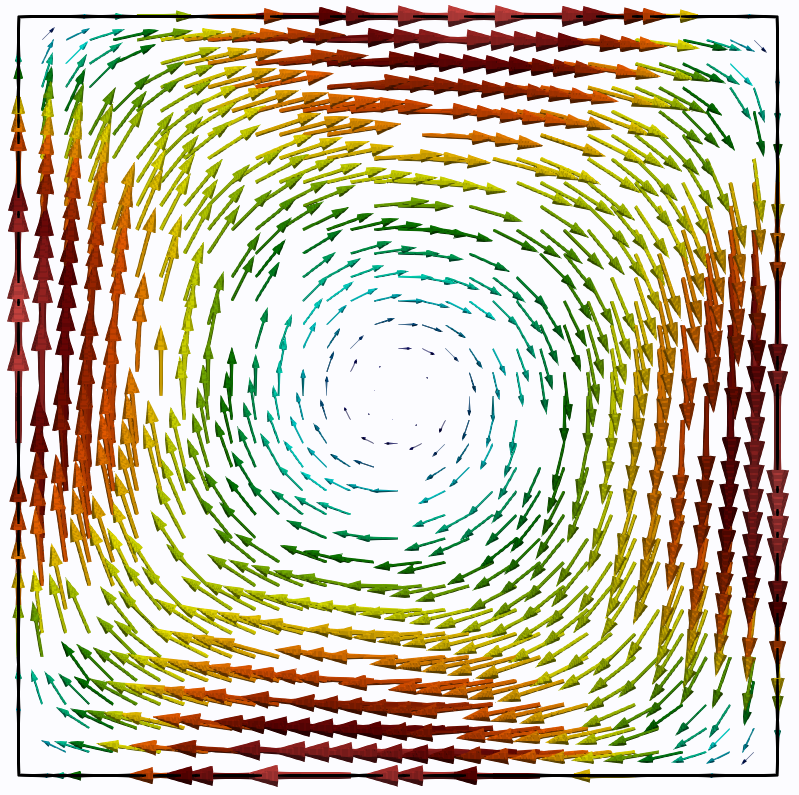}
	\caption{The computed velocity field\fch{ for Example 1}.} 
	\label{Fig2}
\end{figure}

\begin{table}[H]
\center
\begin{tabular}{c||c|c||c|c||c|c}
%
$h$ & $\|p-p_h\|_{0,\Omega}$ & order & $\|\bfu-\bfu_h\|_{0,\Omega}$ & order & $|\bfu-\bfu_h|_{1,\Omega}$ & order \\ \hline\hline
0.353553&0.256600&   ---&   0.055039&   ---&   1.058715&   ---\\ 
0.176777& 0.110749& 1.21& 0.017263& 1.67& 0.538051& 0.97\\ 
0.088388& 0.040998& 1.43& 0.004827& 1.83& 0.270114& 0.99\\ 
0.044194& 0.014566& 1.49& 0.001276& 1.91& 0.135161& 0.99\\ 
0.022097& 0.005134& 1.50& 0.000328& 1.96& 0.067574& 1.00\\ 

\end{tabular}
\caption{Approximation errors and convergence orders for each variable.}
\label{tableEx1}
\end{table}

\begin{table}[H]
\centering
\begin{tabular}{l|lll||lll}
                        & \multicolumn{3}{c||}{$\theta = -1$}                                                                                 & \multicolumn{3}{c}{$\theta = 1$}                                                                                  \\ \hline
\multicolumn{1}{c|}{$h$} & \multicolumn{1}{c|}{$\gamma_0=10^{-3}$} & \multicolumn{1}{c|}{$\gamma_0=1$} & \multicolumn{1}{c||}{$\gamma_0=10^3$} & \multicolumn{1}{c|}{$\gamma_0=10^{-3}$} & \multicolumn{1}{c|}{$\gamma_0=1$} & \multicolumn{1}{c}{$\gamma_0=10^3$} \\ 
\multicolumn{1}{l|}{0.353553}  & \multicolumn{1}{l|}{0.233603}                   & \multicolumn{1}{l|}{0.187756}             &            0.001221                          & \multicolumn{1}{l|}{0.182408}                   & \multicolumn{1}{l|}{0.158295}             &             0.001222                         \\ 
\multicolumn{1}{l|}{0.176777}  & \multicolumn{1}{l|}{0.043670}                   & \multicolumn{1}{l|}{0.035254}             &              0.000250                        & \multicolumn{1}{l|}{0.039551}                   & \multicolumn{1}{l|}{0.032317}             &              0.000250                       \\ 
\multicolumn{1}{l|}{0.088388}  & \multicolumn{1}{l|}{0.008092}                   & \multicolumn{1}{l|}{0.006591}             &              0.000050                        & \multicolumn{1}{l|}{0.007483}                   & \multicolumn{1}{l|}{0.006229}             &            0.000050                         \\ 
\multicolumn{1}{l|}{0.044194}  & \multicolumn{1}{l|}{0.001524}                   & \multicolumn{1}{l|}{0.001257}             &               0.000010                       & \multicolumn{1}{l|}{0.001419}                   & \multicolumn{1}{l|}{0.001235}             &                0.000010                      \\ 
\multicolumn{1}{l|}{0.022097}  & \multicolumn{1}{l|}{0.000297}                   & \multicolumn{1}{l|}{0.000250}             &             0.000002                         & \multicolumn{1}{l|}{0.000280}                   & \multicolumn{1}{l|}{0.000256}             &               0.000002                       \\ 
\end{tabular}
\caption{Computations of $\|\bfu_h\cdot\bfn\|_{0,\Gamma_S}$ for different values of $\theta$ and $\gamma_0$.} \label{tab2}
\end{table}

%
%
%

\clearpage

\subsection{Example 2: 2D Naca 0012}
In this case we use the standard Naca 0012 configuration \fch{depicted in Figure~\ref{Fig5}}. 
The boundary conditions are $\bfu = (51.4814,0)$ on all the box boundaries and $\bfu\cdot\bfn=0$ and $\sigma(\bfu,p)\bfn\cdot \bft=0$ on the surface of the Naca domain. We use a mesh with 8,808 elements with polynomials of order 1, and { we set $\theta= -1$, $\gamma_0=10$}. \fch{The predicted pressure and velocity are depicted Figure~\ref{Fig7} and Figure~\ref{Fig8}, which show the method behaves as expected on this classical example, notably for the enforcement on the slip condition on the wing.}

\begin{figure}[!t]
	\centering
	\includegraphics[width=10.0cm]{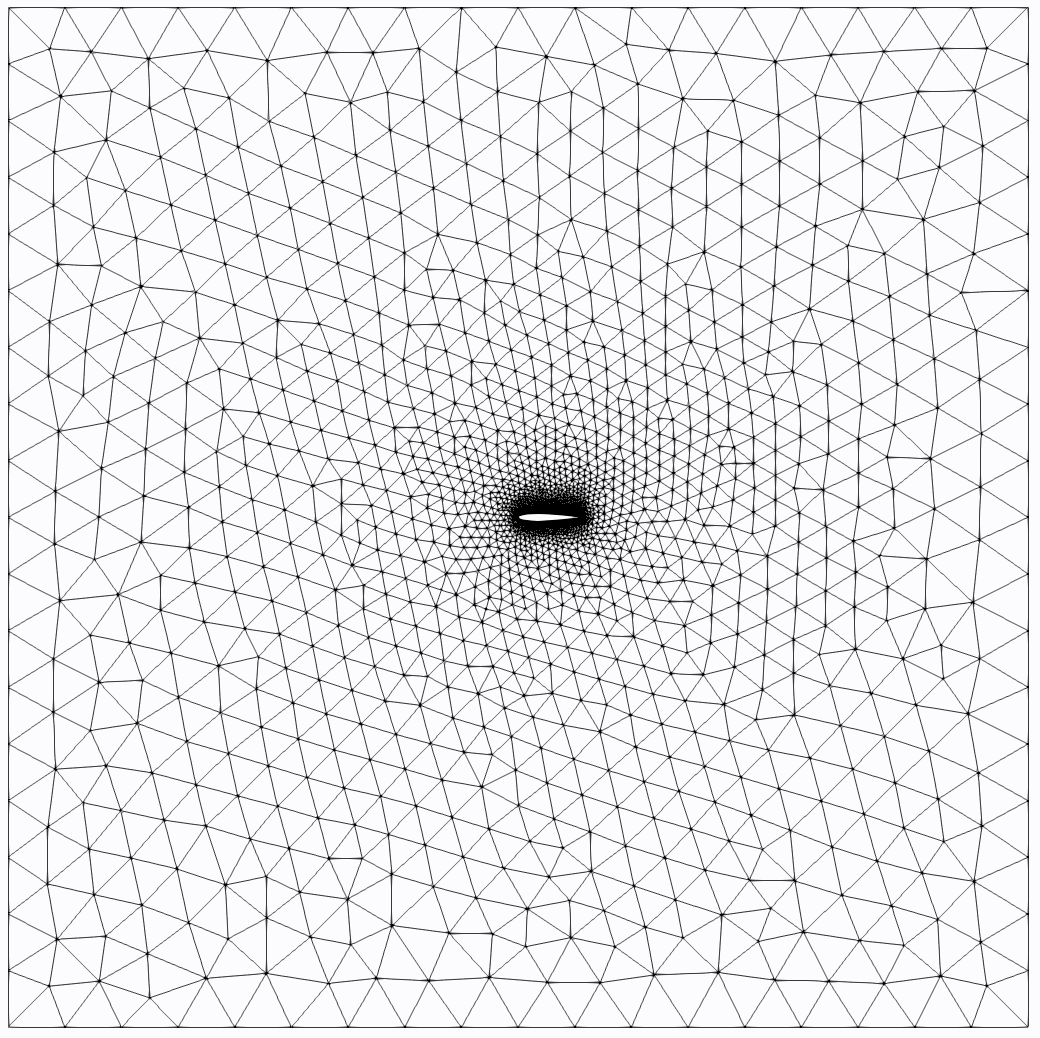}
	\caption{Computational mesh used for the Naca problem.} 
	\label{Fig5}
\end{figure}

\begin{figure}[H]
	\centering
	\includegraphics[width=8.0cm,height=6.cm]{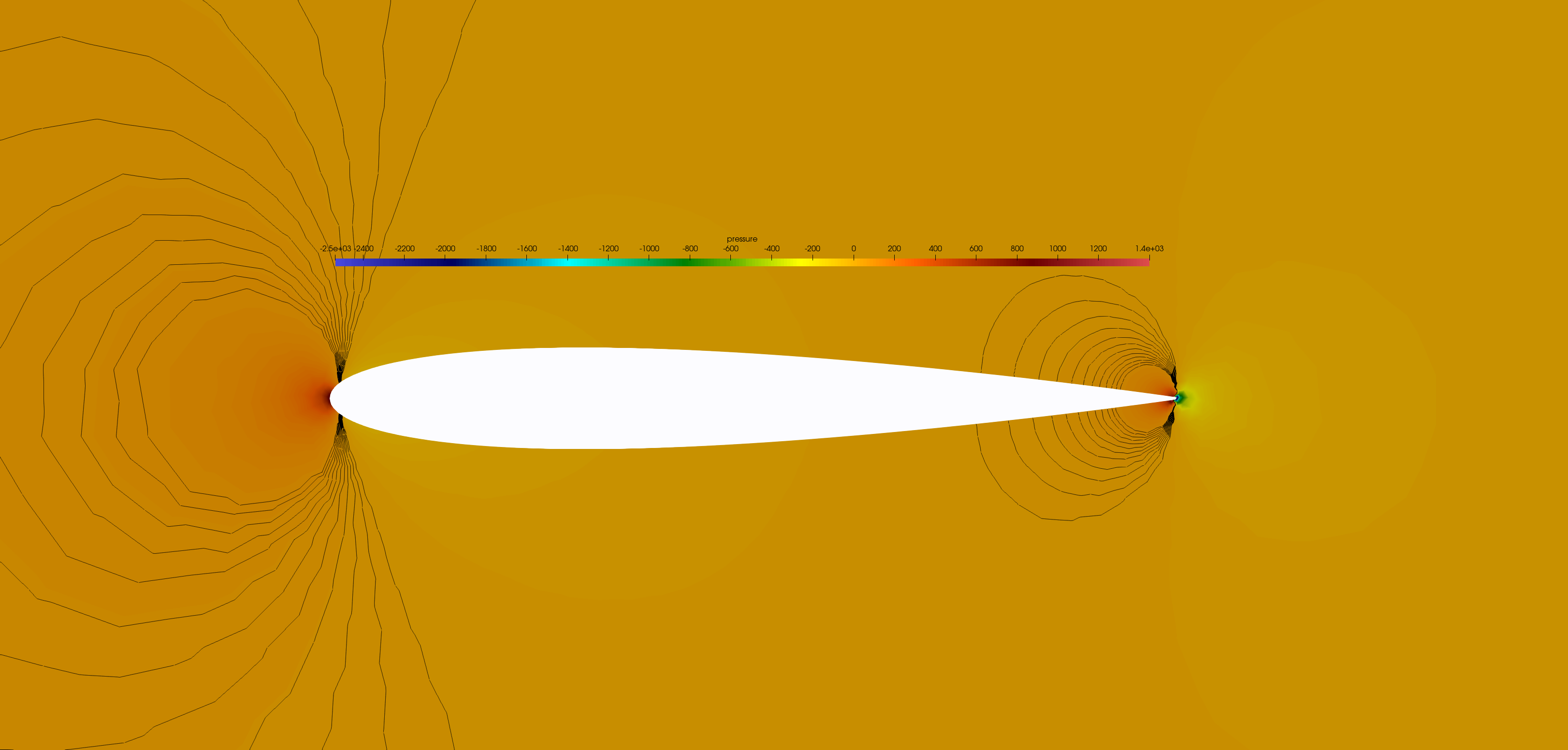}
	\includegraphics[width=8.0cm,height=6.cm]{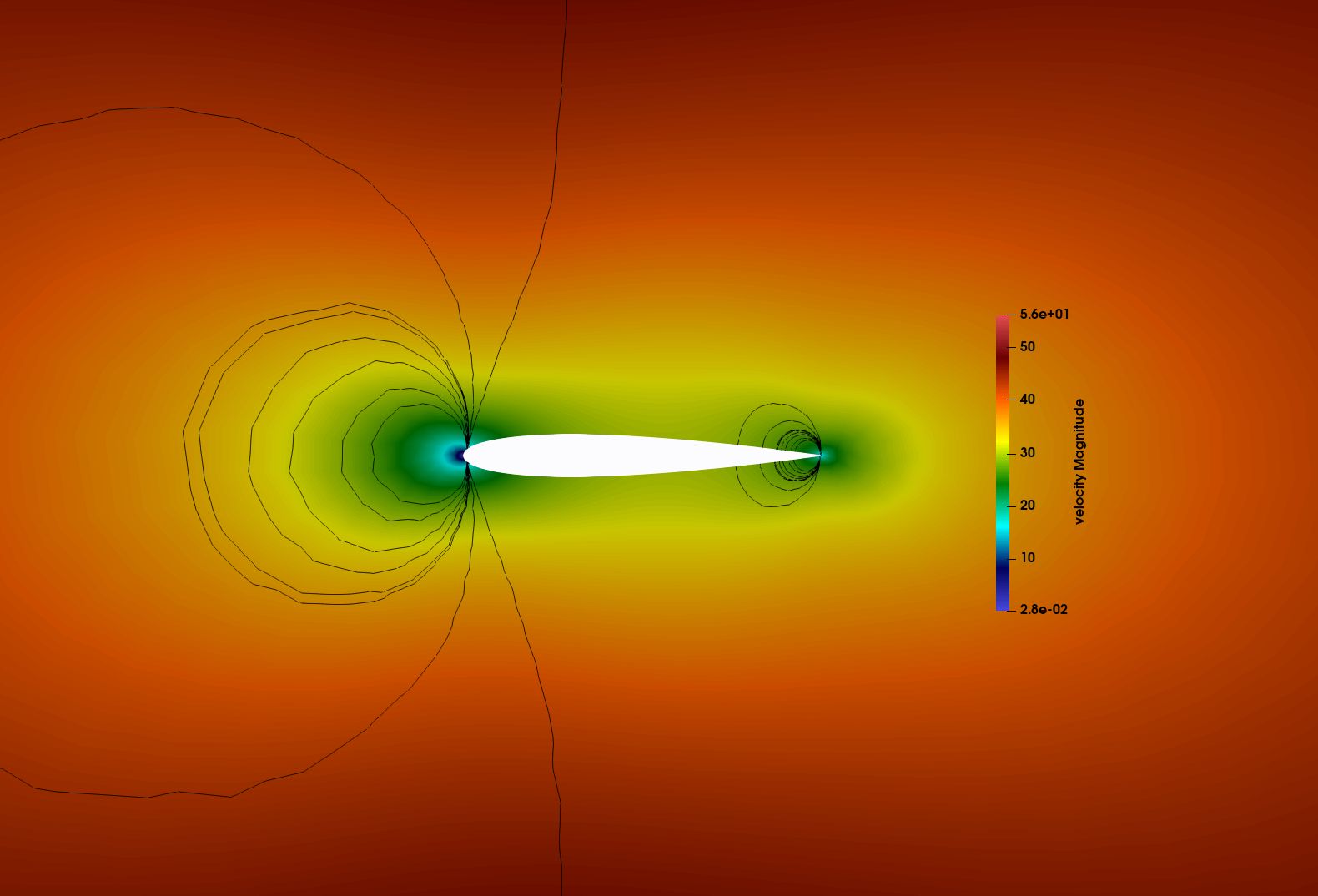}
	\caption{Isovalues of the pressure (left) and velocity magnitude (right).} 
	\label{Fig7}
\end{figure}

\begin{figure}[H]
	\centering
 \includegraphics[width=12.0cm]{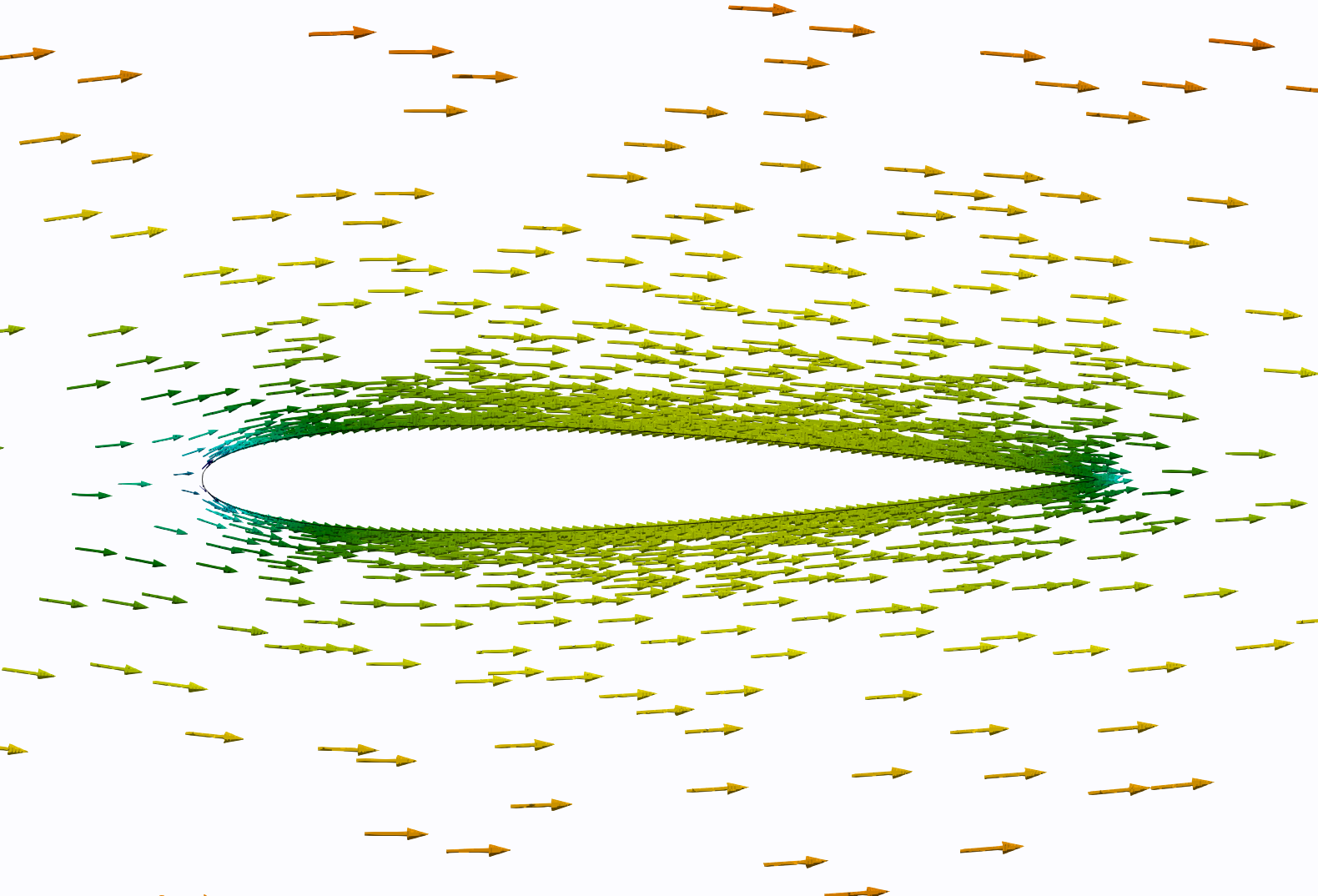}
	\caption{Zoom, close to the Naca wing, of the velocity field.} 
	\label{Fig8}
\end{figure}

\subsection{Example 3: 3D Cylinder}
The last example is based on a standard three-dimensional CFD benchmark: the cylinder problem. The geometrical settings of the domain are given in \fch{Figure \ref{Fig9}}. 
\begin{figure}[H]
	\centering
	\includegraphics[width=15.0cm,height=7.5cm]{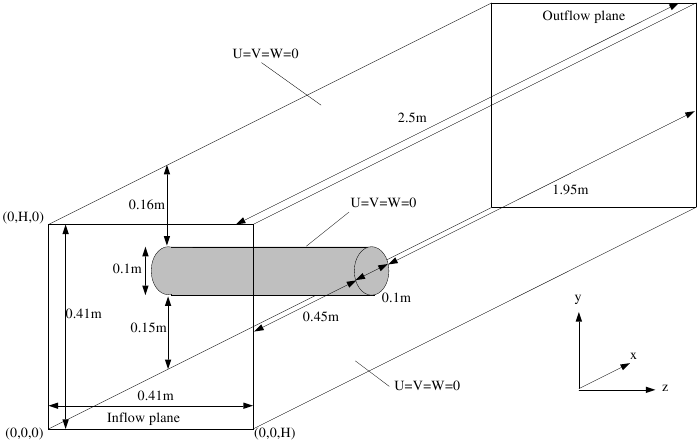}
	\caption{\fch{Cylinder problem.} Domain and boundary conditions.} 
	\label{Fig9}
\end{figure}
In this case $H=0.41\,m$, no-slip boundary conditions are imposed on all the lateral walls of the box, while do-nothing boundary conditions are imposed at the outflow plane. On the surface of the cylinder we impose the conditions $\bfu\cdot\bfn =0$ and $\displaystyle\sum_{i=1}^2\sigma(\bfu,p)\bfn\cdot \bft_i=0$ . Finally, the inflow condition is given by
\[
\bfu_D := \left(\dfrac{16\, U_m\, yz\,(H-y)(H-z)}{H^4}, 0,0\right)^T,
\] 
with $U_m:= 0.45\,m/s$. We use the viscosity $\nu:= 10^{-3}\,m^2/s$. \fch{The mesh is depicted Figure~\ref{Fig10} and { we set $\theta=-1$, $\gamma_0=10$}. The numerical solutions for the pressure and velocity fields are shown in Figure~\ref{Fig11}, Figure~\ref{Fig11b}, Figure~\ref{Fig11c} and Figure~\ref{Fig12}. 
The results illustrate the good behavior of the method
also on this more complex example.}
\begin{figure}[H]
	\centering
	\includegraphics[width=14.0cm,height=4.45cm]{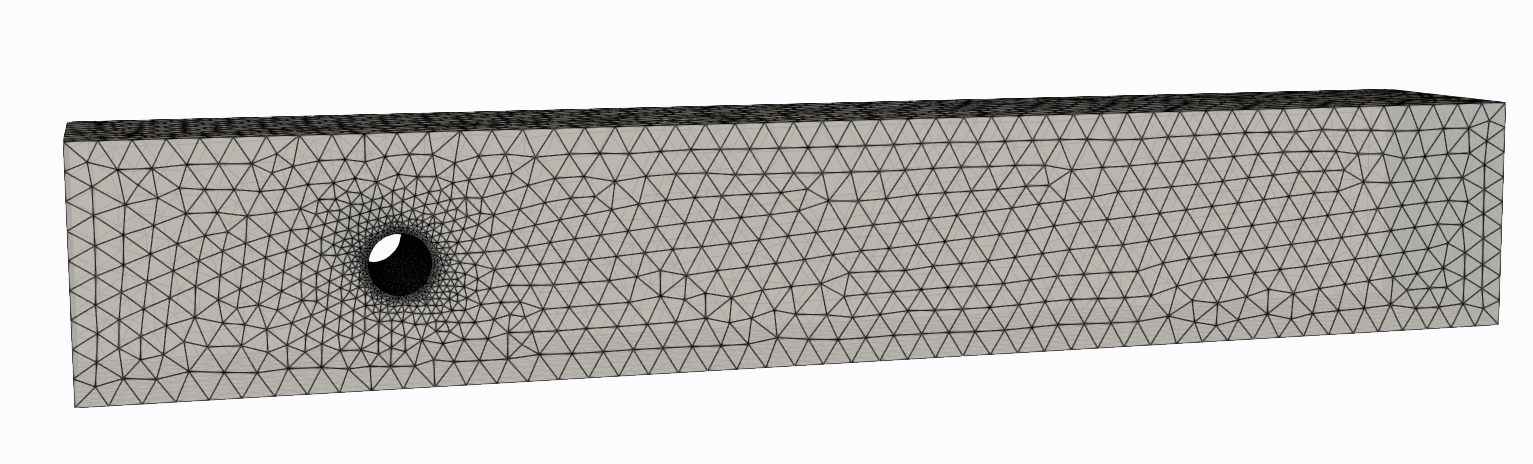}
	\caption{\fch{Cylinder problem.} \alfonso{Surface view of the 
 computational mesh.} }
	\label{Fig10}
\end{figure}

\begin{figure}[H]
	\centering
	\includegraphics[width=12.0cm,height=6.5cm]{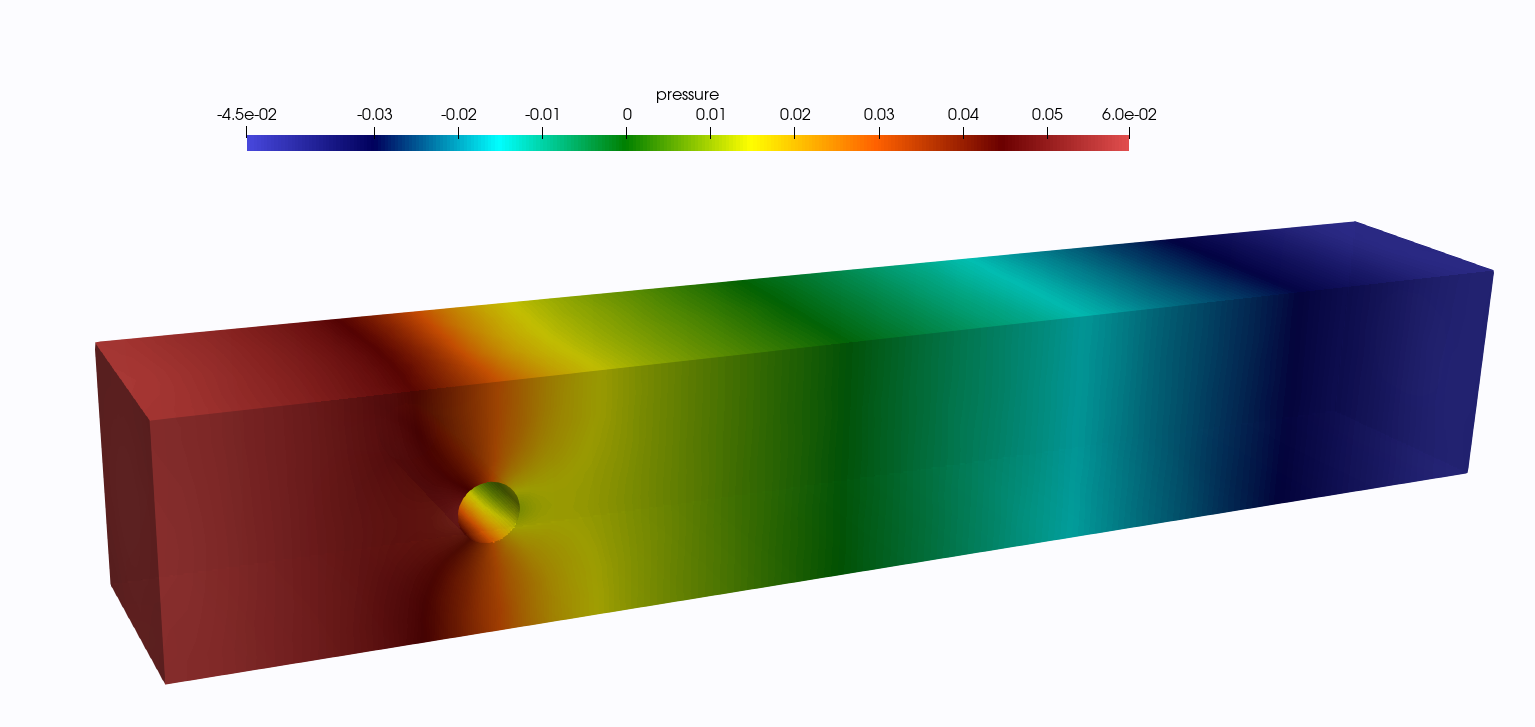}
 \caption{
 \fch{Cylinder problem.} Isovalues of the pressure.} 
	\label{Fig11}
\end{figure}

\begin{figure}[H]
	\centering
	\includegraphics[width=14.0cm,height=6.5cm]{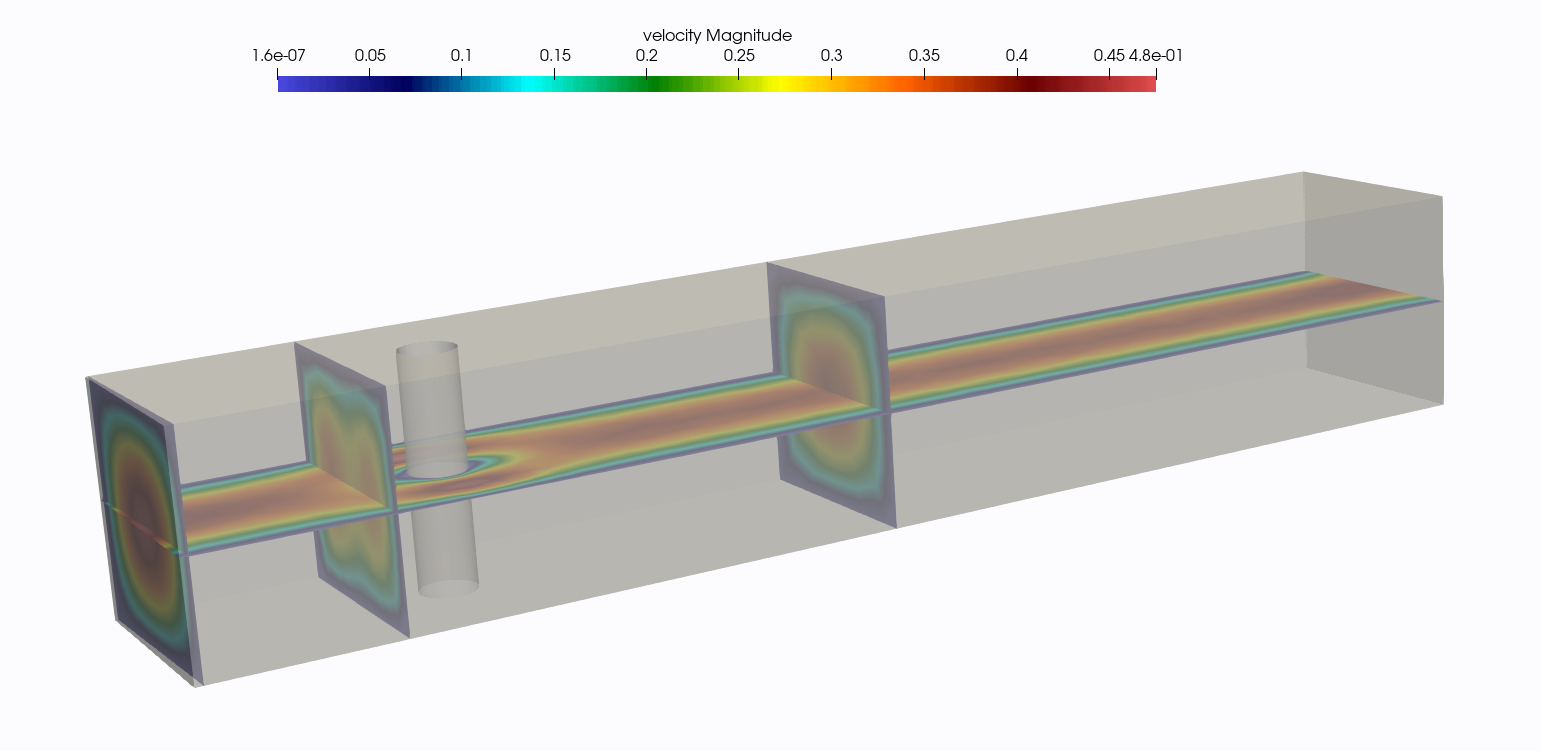}
 \caption{
 \fch{Cylinder problem. Velocity magnitude (top view).}} 
	\label{Fig11b}
\end{figure}

\begin{figure}[H]
	\centering
 \includegraphics[width=14.0cm,height=6.cm]{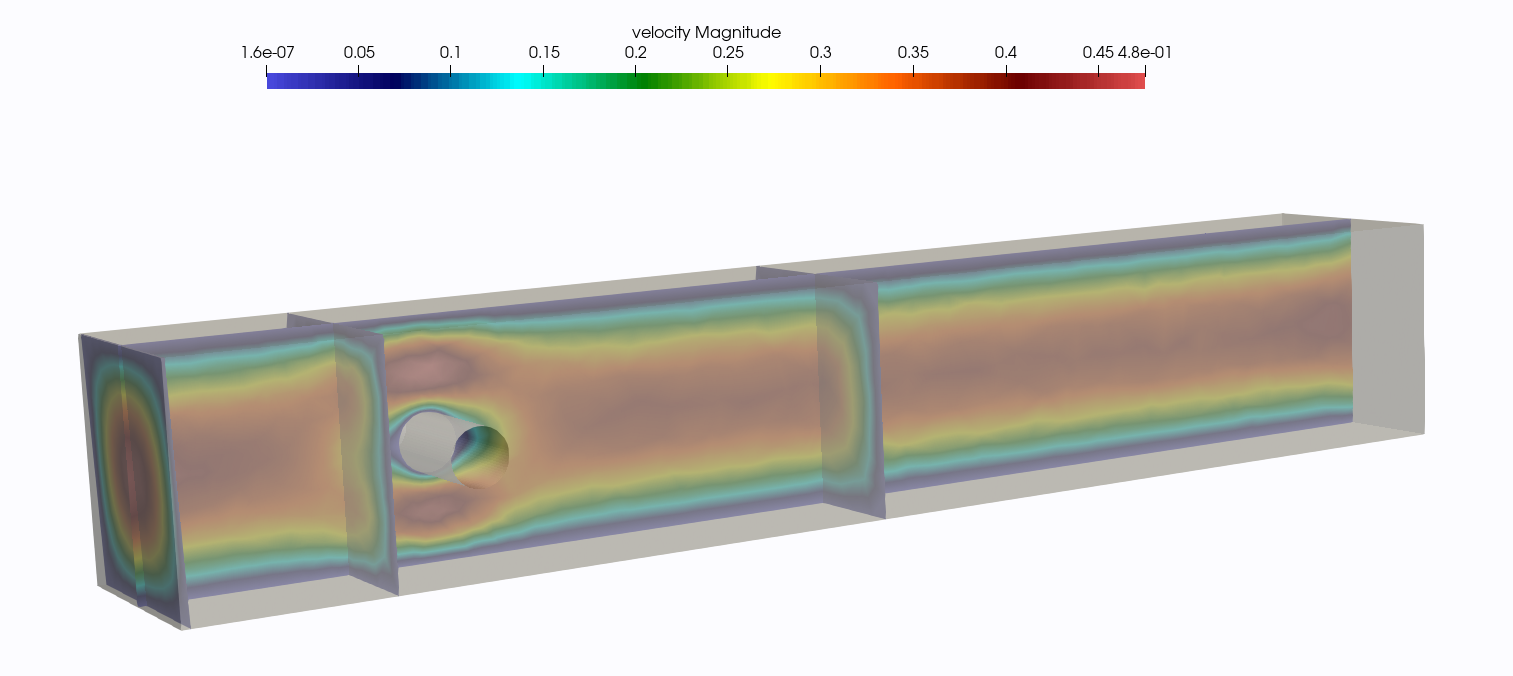}
 \caption{
 \fch{Cylinder problem. Velocity magnitude (side view).}} 
	\label{Fig11c}
\end{figure}

\begin{figure}[H]
	\centering
	\includegraphics[width=11.0cm,height=7.cm]{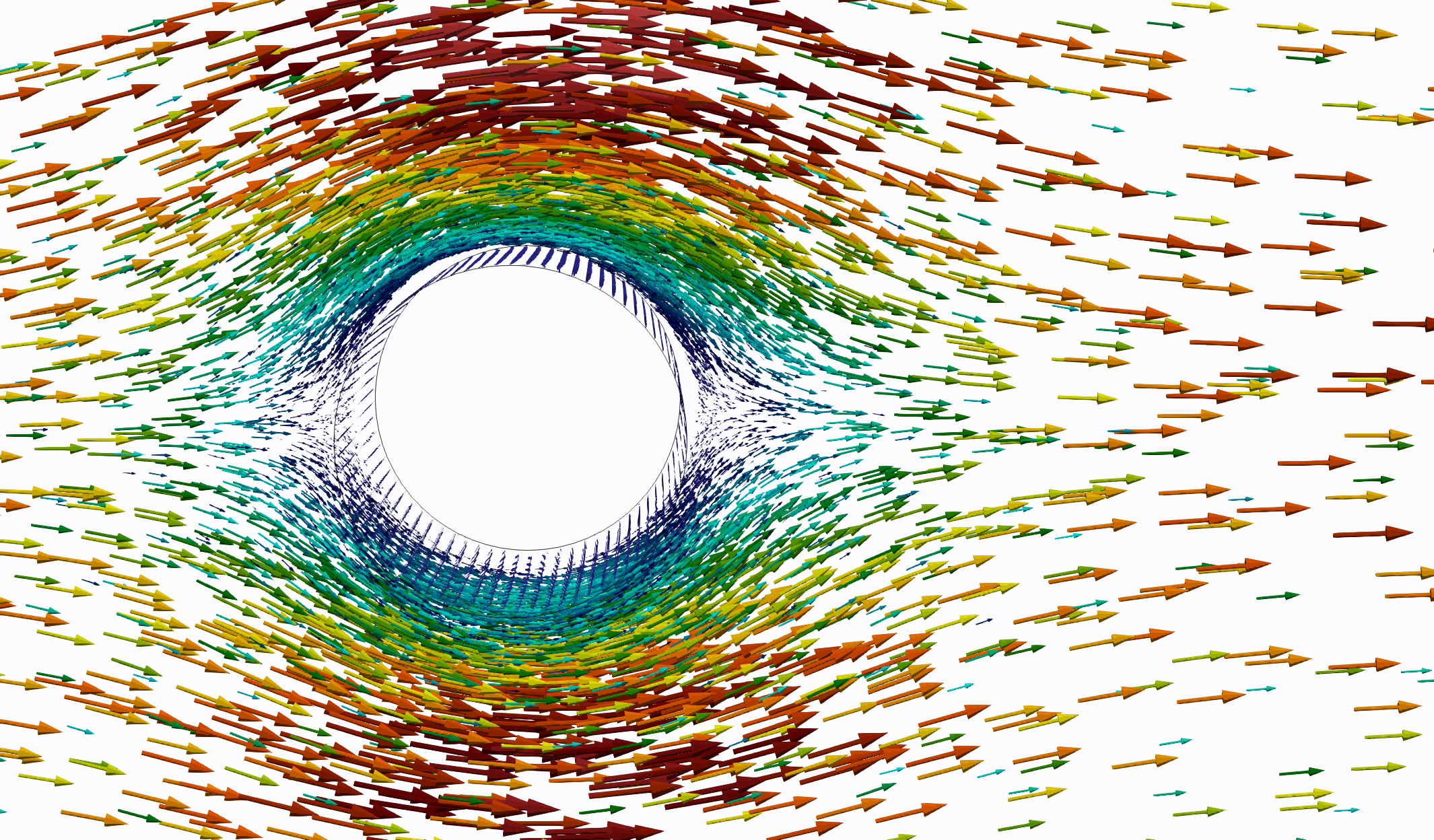}
	\caption{\fch{Cylinder problem.} Zoom of the velocity field (velocity vectors at
 a selection of points) close to the cylinder. The length and the color of the vector is scaled according to velocity magnitude (see the legend in Figure \ref{Fig11c}.} 
	\label{Fig12}
\end{figure}

\section*{Acknowledgements}

R.A. was partially supported by  ANID-Chile through the projects: Centro de Modelamiento Matem\'atico (FB210005) of the PIA Program: Concurso Apoyo a Centros Cient\'ificos y Tecnol\'ogicos de Excelencia con Financiamiento Basal, and Fondecyt Regular No 1211649. 
FC is grateful of the Center for Mathematical Modeling grant FB20005.

\medskip



\end{document}